\documentclass[a4paper]{article}
\usepackage[english]{babel}
\usepackage{amsmath}
\usepackage{amssymb}
\usepackage{wasysym}

\usepackage{tikz}
\usepackage{pgfplots}

\usepackage{graphicx}
\usepackage{mathtools}
\usepackage{amsthm}
\usepackage{booktabs}
\usepackage{xspace}
\usepackage{hyperref}
\usepackage{cleveref}

\usepackage{todonotes}

\definecolor{darkblue}{rgb}{0,0,.5}
\newcommand{\RR}{\mathbb R } 
\newcommand{\XX}{\mathcal{X}} 

\newcommand{\MINLP}{MINLP\xspace}

\newcommand{\MILP}{MILP\xspace}
\newcommand{\xlp}{\bar{x}} 

\DeclareMathOperator{\conv}{conv} 
\DeclareMathOperator{\argmin}{arg\,min} 
\DeclareMathOperator{\midp}{mid} 
\DeclareMathOperator{\dom}{dom} 

\usepackage{xifthen}
\ifthenelse{\isundefined{\T}}{%
  \newcommand{\T}{\mathsf{T}}
  }{%
  \renewcommand{\T}{\mathsf{T}}
}

\theoremstyle{plain}
\newtheorem{theorem}{Theorem}
\newtheorem{proposition}[theorem]{Proposition}
\newtheorem{lemma}[theorem]{Lemma}

\theoremstyle{definition}
\newtheorem{definition}[theorem]{Definition}
\newtheorem{remark}[theorem]{Remark}
\newtheorem{XxmpX}[]{Example}
\newenvironment{example}{\pushQED{\qed}\begin{XxmpX}}{\popQED\end{XxmpX}}

\newcommand{\Title}{\bf Intersection cuts for factorable MINLP}

\usepackage{zibtitlepage}
\ZTPAuthor{
\ZTPHasOrcid{Felipe Serrano}{0000-0002-7892-3951}}
\ZTPTitle{\Title}
\ZTPNumber{18-59}
\ZTPMonth{12}
\ZTPYear{2018}
\ZTPInfo{This work has been supported by the Research Campus MODAL
  \emph{Mathematical Optimization and Data Analysis Laboratories} funded by the
  Federal Ministry of Education and Research (BMBF Grant~05M14ZAM).\\
  The author thank the Schloss Dagstuhl -- Leibniz Center for Informatics for
  hosting the Seminar 18081 "Designing and Implementing Algorithms for
  Mixed-Integer Nonlinear Optimization" for providing the environment to develop
  the ideas in this paper.}

\author{Felipe Serrano\thanks{Zuse Institute Berlin, Takustr.~7, 14195~Berlin,
Germany, \texttt{serrano@zib.de}}}
\title{\Title}

\begin{document}

\zibtitlepage
\maketitle


\begin{abstract}
    Given a factorable function $f$, we propose a procedure that constructs a
    concave underestimor of $f$ that is tight at a given point.
    These underestimators can be used to generate intersection cuts.
    A peculiarity of these underestimators is that they do not rely on a bounded
    domain.
    We propose a strengthening procedure for the intersection cuts that
    exploits the bounds of the domain.
    Finally, we propose an extension of monoidal strengthening to take advantage
    of the integrality of the non-basic variables.

\end{abstract}


\section{Introduction}
\label{sec:intro}

In this work we propose a procedure for generating intersection cuts for mixed
integer nonlinear programs (\MINLP).
We consider \MINLP of the following form
\begin{equation}
    \label{eq:MINLP}
    \begin{aligned}
        \max  \;& c^\T x \\
        \;\text{s.t.}\;& g_j(x) \leq 0, j \in J \\
                       & A x = b \\
                       & x_i \in \mathbb{Z}, i \in I \\
                       & x \geq 0,
    \end{aligned}
\end{equation}
where $J = \{1, \ldots, l\}$ denotes the indices of the nonlinear constraints,
$g_j \colon \mathbb{R}^n \to \mathbb{R}$ are assumed to be continuous and
factorable (see \Cref{def:factorable}), $A \in \mathbb{R}^{m \times n}$, $c \in
\mathbb{R}^n$, $b \in \mathbb{R}^m$, and $I \subseteq \{1, \ldots, n\}$ are the
indices of the integer variables.
We denote the set of feasible solutions by $S$ and a generic relaxation of $S$
by $R$, that is, $S \subseteq R$.
When $R$ is a translated simplicial cone and $C$ contains its apex and no point
of $S$ in its interior, valid inequalities for $\conv(R \setminus C)$ are called
intersection cuts~\cite{Balas1971}.
See the excellent survey~\cite{ConfortiCornuejolsZambelli2011a} for recent
developments and details on intersection cuts for mixed integer linear programs
(\MILP).

Many applications can be modeled as \MINLP~\cite{BoukouvalaMisenerFloudas2016}.
The current state of the art for solving \MINLP to global optimality is via
linear programming (LP), convex nonlinear programming and (\MILP) relaxations of
$S$, together with spatial branch and
bound~\cite{BelottiLeeLibertiMargotWaechter2009,KilincSahinidis2017,LinSchrage2009,MisenerFloudas2014,TawarmalaniSahinidis2005,VigerskeGleixner2017}.
Roughly speaking, the LP-based spatial branch and bound algorithm works as
follows.
The initial polyhedral relaxation is solved and yields $\xlp$.
If the solution $\xlp$ is feasible for~\eqref{eq:MINLP}, we obtain an optimal
solution.
If not, we try to separate the solution from the feasible region.
This is usually done by considering each violated constraint separately.
Let $g(x) \leq 0$ be a violated constraint of~\eqref{eq:MINLP}.
If $g(\xlp) > 0$ and $g$ is convex, then $g(\xlp) + v^\T(x - \xlp) \leq 0$,
where $v \in \partial g(\xlp)$ and $\partial g(\xlp)$ is the subdifferential of
$g$ at $\xlp$, is a valid cut.
If $g_j$ is non-convex, then a convex underestimator $g_{vex}$, that is, a
convex function such that $g_{vex}(x) \leq g(x)$ over the feasible region, is
constructed and if $g_{vex} (\xlp) > 0$ the previous cut is constructed for
$g_{vex}$.
If the point cannot be separated, then we branch, that is, we select a variable
$x_k$ in a violated constraint and split the problem into two problems, one with
$x_k \leq \xlp_k$ and the other one with $x_k \geq \xlp_k$.

Applying the previous procedure to the \MILP case, that is~\eqref{eq:MINLP} with
$J = \emptyset$, reveals a problem with this approach.
In this case, the polyhedral relaxation is just the linear programming (LP)
relaxation.
Assuming that $\xlp$ is not feasible for the \MILP, then there is an $i \in I$
such that $x_i \notin \mathbb{Z}$.
Let us treat the constraint $x_i \in \mathbb{Z}$ as a nonlinear non-convex
constraint represented by some function as $g(x_i) \leq 0$.
Then, $g(\xlp_i) > 0$.
However, a convex underestimator $\bar g$ of $g$ must satisfy that $g_{vex}(z)
\leq 0$ for every $z \in \mathbb{R}$, since $g_{vex}(z) \leq g(z) \leq 0$ for
every $z \in \mathbb{Z}$ and $g_{vex}(z)$ is convex.
Since separation is not possible, we need to branch.

However, for the current state-of-the-art algorithms for \MILP, cutting planes
are a fundamental component~\cite{AchterbergWunderling2013}.
A classical technique for building cutting planes in \MILP is based on
exploiting information from the simplex
tableau~\cite{ConfortiCornuejolsZambelli2011a}.
When solving the LP relaxation, we obtain $x_B = \xlp_B + R x_N$,
where $B$ and $N$ are the indices of the basic and non-basic variables,
respectively.
Since $\xlp$ is infeasible for the \MILP, there must be some $k \in B \cap I$
such that $\xlp_k \notin \mathbb{Z}$.
Now, even though $\xlp$ cannot be separated from the violated constraint $x_k
\in \mathbb{Z}$, the equivalent constraint, $\xlp_k + \sum_{j \in N} r_{kj} x_j
\in \mathbb{Z}$ can be used to separate $\xlp$.

In the \MINLP case, this framework generates equivalent non-linear constraints
with some appealing properties.
The change of variables $x_k = \xlp_k + \sum_{j \in N} r_{kj} x_j$ for the basic
variables present in a violated nonlinear constraint $g(x) \leq 0$, produces the
non-linear constraint $h(x_N) \leq 0$ for which $h(0) > 0$ and $x_N \geq 0$.
Assuming that the convex envelope of $h$ exists in $x_N \geq 0$,
then we can always construct a valid inequality.
Indeed, by~\cite[Corollary 3]{TawarmalaniSahinidis2002}, the convex envelope of
$h$ is tight at 0.
Since an $\epsilon$-subgradient\footnote{
    An $\epsilon$-subgradient of a convex function $f$ at $y \in \dom f$ is $v$
    such that $f(x) \geq f(y) - \epsilon + v^\T(x - y)$ for all $x \in \dom f$
}
always exists for any $\epsilon > 0$ and $x \in
\dom h$~\cite{BrondstedRockafellar1965}, an
$\frac{h(0)}{2}$-subgradient, for instance, at 0 will separate it.

Even when there is no convex underestimator for $h$, a valid cutting plane does
exist.
Continuity of $h$ implies that $X = \{ x_N \geq 0 : h(x_N) \leq 0\}$ is closed
and~\cite[Lemma 2.1]{ConfortiCornuejolsDaniilidisLemarechalMalick2015} ensures
that $0 \notin \overline{conv} X$, thus, a valid inequality exists.
We introduce a technique to construct such a valid inequality.
The idea is to build a \emph{concave underestimator} of $h$, ${h}_{ave}$, such
that ${h}_{ave}(0) = h(0) > 0$.
Then, $C = \{x_N : {h}_{ave}(x_N) \geq 0\}$ is an \emph{$S$-free set}, that is, a
convex set that does not contain any feasible point in its interior, and as such
can be used to build an intersection cut
(IC)~\cite{Tuy1964,Balas1971,Glover1973}.

\paragraph{First contribution}%
In \Cref{sec:under}, we present a procedure to build concave underestimators for factorable
functions that are tight at a given point.
The procedure is similar to McCormick's method for constructing convex
underestimators, and generalizes Proposition 3.2 and improves Proposition 3.3
of~\cite{Khamisov1999}.
These underestimators can be used to build intersection cuts.
We note that IC from a concave underestimator can generate cuts that cannot be
generated by using the convex envelope.
This should not be surprising, given that intersection cuts work at the
feasible region level, while convex underestimators depend on the graph
of the function.
A simple example is $\{x \in [0,2]: -x^2 + 1 \leq 0\}$.
When separating 0, the intersection cut gives $x \geq 1$, while using the
convex envelope over $[0,2]$ yields $x \geq 1/2$.\\
%

There are many differences between concave underestimators and convex ones.
Maybe the most interesting one is that concave underestimators do not need
bounded domains to exist.
As an extreme example, $-x^2$ is a concave underestimator of itself, but a
convex underestimator only exists if the domain of $x$ is bounded.
Even though this might be regarded as an advantage, it is also a problem.
If concave underestimators are independent of the domain, then we cannot improve
them when the domain shrinks.

\paragraph{Second contribution}%
In \Cref{sec:strength}, we propose a strengthening procedure that uses the bounds of the variables to
enlarge the $S$-free set.
Our procedure improves on the one used by Tuy~\cite{Tuy1964}.\\

Other techniques for strengthening IC have been proposed, such as, exploiting
the integrality of the non-basic
variables~\cite{BalasJeroslow1980,ConfortiCornuejolsZambelli2011,DeyWolsey2010},
improving the relaxation $R$~\cite{BalasMargot2011,Porembski1999,Porembski2001}
and computing the convex hull of $R \setminus
C$~\cite{BasuCornuejolsZambelli2011,ConfortiCornuejolsDaniilidisLemarechalMalick2015,Glover1974,SenSherali1986,SenSherali1987}.

\paragraph{Third contribution}%
By interpreting IC as disjunctive cuts~\cite{Balas1979}, we extend monoidal
strengthening to our setting~\cite{BalasJeroslow1980} in \Cref{sec:monoid}.
Although its applicability seems to be limited, we think it is of independent
interest, specially for MILP.\\



\section{Related work}
\label{sec:literature}
There have been many efforts on generalizing cutting planes from \MILP to
\MINLP, we refer the reader to~\cite{ModaresiKilincVielma2015} and the
references therein.
In~\cite{ModaresiKilincVielma2015}, the authors study how to compute
$\conv(R \setminus C)$ where $R$ is not polyhedral, but $C$ is a $k$-branch
split.
In practice, such sets $C$ usually come from the integrality of the variables.
Works that build sets $C$ which do not come from integrality considerations
include~\cite{Belotti2011,BienstockChenMunoz2016,SaxenaBonamiLee2010a,SaxenaBonamiLee2010}.
We refer to~\cite{BonamiLinderothLodi2011} and the references therein for more
details.
We would like to point out that the disjunctions built
in~\cite{Belotti2011,SaxenaBonamiLee2010a,SaxenaBonamiLee2010}
can be interpreted as piecewise linear concave underestimators.
However, our approach is not suitable for disjunctive cuts built through cut
generating LPs~\cite{BalasCeriaCornuejols1993}, since we generate infinite
disjunctions, see \Cref{sec:monoid}, so we rely on the classical concept of
intersection cuts where $R$ is a translated simplicial cone.

Khamisov~\cite{Khamisov1999} studies functions $f : \mathbb{R}^n \to
\mathbb{R}$, representable as $f(x) = \max_{y \in R} \varphi(x,y)$ where
$\varphi$ is continuous and concave on $x$.
These functions allow for a concave underestimator at \emph{every} point.
He shows that this class of functions is very general, in particular, the class
of functions representable as difference of convex functions is a strict subset
of this class.
He then shows how to build concave underestimators of some functions.
The technique in~\cite{Khamisov1999} for building an underestimator for the
composition of two functions is a special case of $\Cref{thm:composition}$
below, and the one for building an underestimator for the product requires a
compact domain.
We simplify the construction for the product and no longer need a compact
domain.

Although not directly related to this work, other papers that use underestimators
other than convex
are~\cite{BuchheimDAmbrosio2016,BuchheimTraversi2013,Hasan2018}.


\section{Concave underestimators}
\label{sec:under}
In his seminal paper~\cite{McCormick1976}, McCormick proposed a method to build
convex underestimators of \emph{factorable} functions.
\begin{definition}
    \label{def:factorable}
    Given a set of univariate functions $\mathcal{L}$, e.g.,$\mathcal{L} =
    \{\cos, \cdot^n, \exp, \log, \text{...}\}$, the set of \emph{factorable
    functions} $\mathcal{F}$ is the smallest set that contains $\mathcal{L}$,
    the constant functions, and is closed under addition, product and
    composition.
\end{definition}
As an example, $e^{-(\cos(x^2) + xy/4)^2}$ is a factorable function for
$\mathcal{L} = \{\cos, \exp\}$.

Given the inductive definition of factorable functions, to show a property about
them one just needs to show that said property holds for all the functions in
$\mathcal{L}$, constant functions, and that it is preserved by the product,
addition and composition.
For instance, McCormick~\cite{McCormick1976} proves, constructively, that every
factorable function admits a convex underestimator and a concave overestimator,
by showing how to construct estimators for the sum, product and composition of
two functions for which estimators are known.

An estimator for the sum of two functions is the sum of the
estimators.
For the product, McCormick uses the well-known McCormick inequalities.
Less known is the way McCormick handles the composition $f(g(x))$.
Let $f_{vex}$ be a convex underestimator of $f$ and $z_{\min} = \argmin
f_{vex}(z)$.
Let $g_{vex}$ be a convex underestimator of $g$ and $g^{ave}$ a concave
overestimator.
McCormick shows\footnote{He actually leaves it as an exercise for the
reader.} that $f_{vex} (\midp\{ g_{vex}(x), g^{ave}(x), z_{\min} \})$ is a convex
underestimator of $f(g(x))$, where $\midp\{x,y,z\}$ is the median between $x,y$
and $z$.
It is well known that the optimum of a convex function over a closed interval is
given by such a formula, thus
\[
    f_{vex} (\midp \{ g_{vex}(x), g^{ave}(x), z_{\min} \}) =
    \min \{ f_{vex}(z) : z \in [g_{vex}(x), g^{ave}(x)] \},
\]
see also~\cite{TsoukalasMitsos2014}.

\begin{definition}
    Let $\XX \subseteq \mathbb{R}^n$ be convex, and $f : \XX \to \RR$ be a
    function.
    We say that $f_{ave} : \XX \to \RR$ is a \emph{concave underestimator of $f$
    at} $\xlp \in \XX$ if $f_{ave}(x)$ is concave, $f_{ave}(x) \leq f(x)$ for
    every $x \in \XX$ and $f_{ave}(\xlp) = f(\xlp)$.
    Similarly we define a \emph{convex overestimator of $f$ at} $\xlp \in \XX$.
\end{definition}
\begin{remark}
    For simplicity, we will consider only the case where $\XX = \mathbb{R}^n$.
    This restriction leaves out some common functions like $\log$.
    One possibility to include these function is to let the range of the
    function to be $\mathbb{R}\cup\{\pm \infty\}$.
    Then, $\log(x) = -\infty$ for $x \in \mathbb{R}_-$.
    Note that other functions like $\sqrt{x}$ can be handled by replacing them
    by a concave underestimator defined on all $\mathbb{R}$.
\end{remark}
We now show that every factorable function admits a concave underestimator at a
given point.
Since the case for the addition is easy, we just need to specify how to build
concave underestimators and convex overestimators for
\begin{itemize}
    \item the product of two functions for which estimators are known,
    \item the composition $f(g(x))$ where estimators of $f$ and $g$ are known
        and $f$ is univariate.
\end{itemize}

\begin{theorem}
    \label{thm:composition}
    Let $f : \mathbb{R} \to \mathbb{R}$ and $g : \mathbb{R}^n \to \mathbb{R}$.
    Let $g_{ave}, f_{ave}$ be, respectively, a concave underestimator
    of $g$ at $\xlp$ and of $f$ at $g(\xlp)$.
    Further, let $g^{vex}$ a convex overestimator of $g$ at $\xlp$.
    Then,
    \[
        h(x) := \min \{ f_{ave}(g_{ave}(x)), f_{ave}(g^{vex}(x)) \},
    \]
    is a concave underestimator of $f(g(x))$ at $\xlp$.
\end{theorem}
\begin{proof}
    Clearly, $h(\xlp) = f(g(\xlp))$.

    To establish $h(x) \leq f(g(x))$, notice that
    \begin{equation}
        \label{eq:niceform}
        h(x) = \min \{ f_{ave}(z) : g_{ave}(x) \leq z \leq g^{vex}(x) \}.
    \end{equation}
    Since $z = g(x)$ is a feasible solution and $f_{ave}$ is an underestimator
    of $f$, we obtain that $h(x) \leq f(g(x))$.

    Now, let us prove that $h$ is concave.
    To this end, we again use the representation~\eqref{eq:niceform}.
    To simplify notation, we write $g_1, g_2$ for $g_{ave}, g^{vex}$,
    respectively.\\
    We prove concavity by definition, that is,
    \[
        h(\lambda x_1 + (1 - \lambda) x_2) \geq \lambda h(x_1) + (1 - \lambda)
        h(x_2), \text{ for } \lambda \in [0,1].
    \]
    Let
    \begin{align*}
        I &= [g_1(\lambda x_1 + (1 - \lambda) x_2), g_2(\lambda x_1 + (1 -
        \lambda) x_2)]\\
        J &= [\lambda g_1(x_1) + (1 - \lambda) g_1(x_2), \lambda g_2(x_1) + (1 -
        \lambda) g_2(x_2)].
    \end{align*}
    By the concavity of $g_1$ and convexity of $g_2$ we have $I \subseteq J$.
    Therefore,
    \begin{align*}
    h(\lambda x_1 + (1 - \lambda) x_2) = \min \{ f_{ave}(z) : z \in I \}
                                       \geq \min \{ f_{ave}(z) : z \in J \}.
    \end{align*}
    Since $f_{ave}$ is concave, the minimum is achieved at the boundary,
    \[
        \min \{ f_{ave}(z) : z \in J \} =
        \min_{i \in \{1, 2\}} f_{ave}(\lambda g_i(x_1) + (1 - \lambda)
        g_i(x_2)).
    \]
    Furthermore, $f_{ave}(\lambda g_i(x_1) + (1 - \lambda) g_i(x_2)) \geq
    \lambda f_{ave}(g_i(x_1)) + (1 - \lambda) f_{ave}(g_i(x_2))$
    which implies that
    \begin{align*}
        h(\lambda x_1 + (1 - \lambda) x_2) &\geq
        \min_{i \in \{1, 2\}}
        \lambda f_{ave}(g_i(x_1)) + (1 - \lambda) f_{ave}(g_i(x_2)) \\
        &\geq \min_{i \in \{1, 2\}}
        \lambda f_{ave}(g_i(x_1)) +
        \min_{i \in \{1, 2\}} (1 - \lambda) f_{ave}(g_i(x_2)) \\
        &= \lambda h(x_1) + (1 - \lambda) h(x_2),
    \end{align*}
    as we wanted to show.
\end{proof}

\begin{remark}
    The generalization of \Cref{thm:composition} to the case where $f$ is
    multivariate in the spirit of~\cite{TsoukalasMitsos2014} is straightforward.
\end{remark}

The computation of a concave underestimator and convex overestimator of the
product of two functions reduces to the computation of estimators for the square
of a function through the polarization identity
\[
    4 f(x) g(x) = (f(x) + g(x))^2 - (f(x) - g(x))^2.
\]
Let $h : \mathbb{R}^n \to \mathbb{R}$ for which we know estimators
$h_{vex} \leq h \leq h^{ave}$ at $\xlp$.
From \Cref{thm:composition}, a convex overestimator of $h^2$ at $\xlp$ is given
by $\max\{{h_{vex}}^2, {h^{ave}}^2\}$.
On the other hand, a concave underestimator of $h^2$ at $\xlp$ can be constructed
from the underestimator $h^2(x) \geq h^2(\xlp) + 2 h(\xlp)(h(x) - h(\xlp))$.
From here we obtain
\begin{equation}
    \label{eq:product}
    \begin{cases}
        2 h(\xlp)h^{ave}(x) - h^2(\xlp), &\text{ if } h(\xlp) \leq 0\\
        2 h(\xlp)h_{vex}(x) - h^2(\xlp), &\text{ if } h(\xlp) > 0.
    \end{cases}
\end{equation}

\begin{example}
    Let us compute a concave underestimator of $f(x) = e^{-(\cos(x^2) + x/4)^2}$
    at $0$.
    Estimators of $x^2$ are given by $0 \leq x^2 \leq x^2$.
    For $\cos(x)$, estimators are $\cos(x) - x^2/2 \leq \cos(x) \leq 1$.
    Then, a concave underestimator of $\cos(x^2)$ is, according to
    \Cref{thm:composition}, $\min\{\cos(0) - 0^2/2, \cos(x^2) - x^4/2\} =
    \cos(x^2) - x^4/2$.
    A convex overestimator is 1.
    Hence, $\cos(x^2) - x^4/2 + x/4 \leq \cos(x^2) + x/4 \leq 1 + x/4$.

    Given that $-x^2$ is concave, a concave underestimator of $-(\cos(x^2) +
    x/4)^2$ is $\min\{-(\cos(x^2) - x^4/2 + x/4)^2, -(1 + x/4)^2\}$.
    To compute a convex overestimator of $-(\cos(x^2) + x/4)^2$, we compute
    a concave underestimator of $(\cos(x^2) + x/4)^2$.
    Since, $\cos(x^2) + x/4$ at 0 is $1$, \eqref{eq:product} yields
    $2 (\cos(x^2) - x^4/2 + x/4) - 1$.

    Finally, a concave underestimator of $e^x$ at $x = -1$ is just its
    linearization, $e^{-1} + e^{-1}(x + 1)$ and so
    $e^{-1} + e^{-1}(1 + \min\{-(\cos(x^2) - x^4/2 + x/4)^2, -(1 + x/4)^2\})$
    is a concave underestimator of $f(x)$.
    The intermediate estimators as well as the final concave underestimator
    are illustrated in \Cref{fig:ex1}.
    \begin{figure}
        \includegraphics[width=\textwidth]{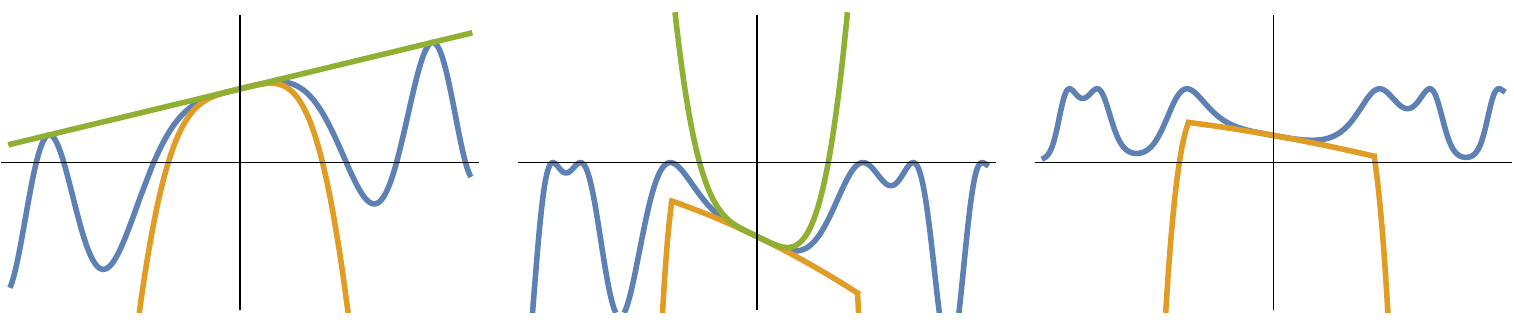}
        \caption{Concave underestimator (orange) and convex overestimator
            (green) of $\cos(x^2) + x/4$ (left), $-(\cos(x^2) +
            x/4)^2$ (middle) and $f(x)$ (right) at $x = 0$.
        } \label{fig:ex1}
    \end{figure}
\end{example}

For ease of exposition, in the rest of the paper we assume that the concave
underestimator is differentiable.
All results can be extended to the case where the functions are only sub- or
super-differentiable.


\section{Enlarging the $S$-free sets by using bound information}
\label{sec:strength}
In \Cref{sec:under}, we showed how to build concave underestimators which give
us $S$-free sets.
Note that the construction does not make use of the bounds of the domain.
We can exploit the bounds of the domain by the observation that the concave
underestimator only needs to underestimate within the feasible region.
However, to preserve the convexity of the $S$-free set, we must ensure that the
underestimator is still concave.

Let $h(x) \leq 0$ be a constraint of \eqref{eq:MINLP}, assume $x \in [l,u]$
and let $h_{ave}$ be a concave underestimator of $h$.
Throughout this section, $S = \{ x \in [l,u] : h(x) \leq 0\}$.
In order to construct a concave function $\hat h$ such that $\{ x : \hat
h(x) \geq 0 \}$ contains $\{ x : h_{ave}(x) \geq 0 \}$, consider the
following function
\begin{equation}
    \label{eq:strong}
    \hat h(x) = \min \{ h_{ave}(z) +
    \nabla h_{ave}(z)^\T (x - z) :\ z \in [l,u],\ h_{ave}(z) \geq 0 \}.
\end{equation}
A similar function was already considered by Tuy~\cite{Tuy1964}.
The only difference is that Tuy's strengthening does not use the restriction
$h_{ave}(z) \geq 0$, see \Cref{fig:ex2}.

\begin{proposition}
    \label{prop:strong}
    Let $h_{ave}$ be a concave underestimator of $h$ at $\xlp \in [l, u]$, such
    that $h(\xlp) > 0$.
    Define $\hat h$ as in \eqref{eq:strong}.
    Then, the set $C = \{x : \hat h(x) \geq 0\}$ is a convex $S$-free set
    and $C \supseteq \{ x : h_{ave}(x) \geq 0 \}$.
\end{proposition}
\begin{proof}
    The function $\hat h$ is concave since it is the minimum of linear
    functions.
    This establishes the convexity of $C$.

    To show that $C \supseteq \{ x : h_{ave}(x) \geq 0 \}$, notice that
    $h_{ave}(x) = \min_{z} h_{ave}(z) + \nabla h_{ave}(z)^\T (x - z)$.
    The inclusion follows from observing that the objective function in the
    definition of $\hat h(x)$ is the same as above, but over a smaller domain.

    To show that it is $S$-free, we will show that for every $x \in [l,u]$ such
    that $h(x) \leq 0$, $\hat h(x) \leq 0$.

    Let $x_0 \in [l,u]$ such that $h(x_0) \leq 0$.
    Since $h_{ave}$ is a concave underestimator at $\xlp$, $h_{ave}(\xlp) > 0$
    and $h_{ave}(x_0) \leq 0$.
    If $h_{ave}(x_0) = 0$, then, by definition, $\hat h(x_0) \leq h_{ave}(x_0) =
    0$ and we are done.
    We assume, therefore, that $h_{ave}(x_0) < 0$.

    Consider $g(\lambda) = h_{ave}(\xlp + \lambda(x_0 - \xlp))$ and let
    $\lambda_1 \in (0,1)$ be such that $g(\lambda_1) = 0$.
    The existence of $\lambda_1$ is justified by the continuity of $g$, $g(0) >
    0$ and $g(1) < 0$.
    Equivalently, $x_1 = \xlp + \lambda_1(x_0 - \xlp)$ is the intersection point
    between the segment joining $x_0$ with $\xlp$ and $\{x : h_{ave}(x) = 0\}$.
    The linearization of $g$ at $\lambda_1$ evaluated at $\lambda = 1$ is
    negative, because $g$ is concave, and equals $h_{ave}(x_1) + \nabla
    h_{ave}(x_1)^T (x_0 - x_1)$.
    Finally, given that $x_1 \in [l,u]$ and $h_{ave}(x_1) = 0$, $x_1$ is
    feasible for \eqref{eq:strong} and we conclude that $\hat h(x_0) < 0$.
\end{proof}

\begin{figure}
    \centering
    \includegraphics[width=\textwidth]{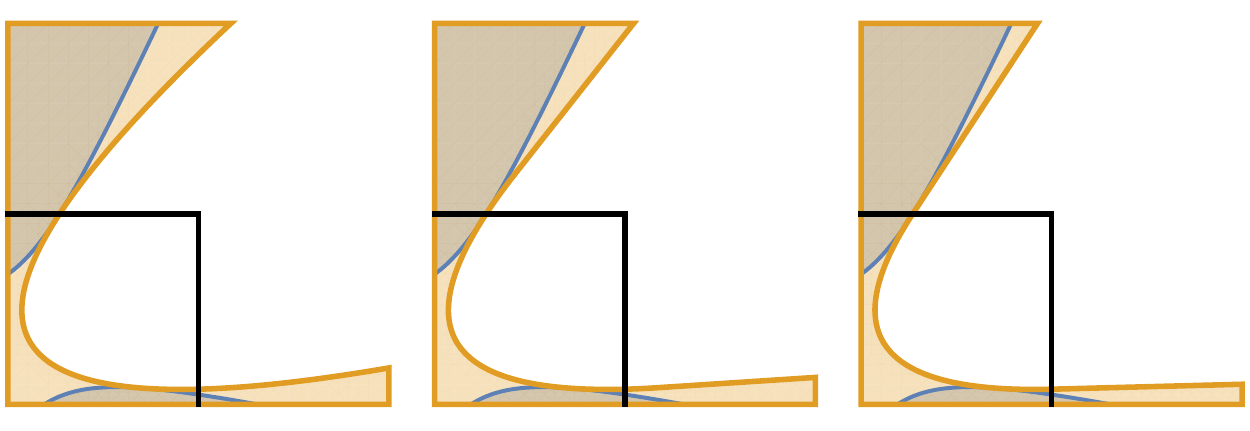}
    \caption{Feasible region $\{x,y \in [0,2] : h(x,y) \leq 0\}$, where
        $h = x^2 - 2y^2 + 4xy -3x+2y+1$, in blue together with
        $h_{ave}(x,y) \leq 0$ at $\xlp = (1,1)$ (left), Tuy's strengthening
        (middle) and $\hat h \leq 0$ (right) in orange. Region shown is
        $[0,4]^2$, $[0,2]^2$ is bounded by black lines.
    } \label{fig:ex2}
\end{figure}

%
%


\section{``Monoidal'' strengthening}
\label{sec:monoid}
We show how to strengthen cuts from reverse convex constraints when exactly
\emph{one} non-basic variable is integer.
Our technique is based on monoidal strengthening applied to disjunctive cuts,
see \Cref{lemma:strength} and the discussion following it.
If more than one variable is integer, we can generate one cut per integer
variable, relaxing the integrality of all but one variable at a time.
However, under some conditions (see \Cref{rmk:more_integers}), we can exploit
the integrality of several variables at the same time.
Our exposition of the monoidal strengthening technique is slightly different
from~\cite{BalasJeroslow1980} and is inspired by~\cite[Section
4.2.3]{Wiese2016}.

Throughout this section, we assume that we already have a concave
underestimator, and that we have performed the change of variables described in
the introduction.
Therefore, we consider the constraint $\{x \in [0, u] : h(x) \leq 0\}$ where
$h : \mathbb{R}^n \to \mathbb{R}$ is concave and $h(0) > 0$.
Let $Y = \{y \in [0, u]: h(y) = 0\}$.
%
The convex $S$-free set $C = \{ x \in [0,u] : h(x) \geq 0\}$ can be written as
\begin{align*}
    C
    &= \bigcap_{y \in Y} \{ x \in [0,u] : \nabla h(y)^\T x \geq \nabla h(y)^\T y
    \}.
\end{align*}
The concavity of $h$ implies that $h(0) \leq h(y) - \nabla h(y)^\T y$ for all $y$
in the domain of $h$.
In particular, if $y \in Y$, then $\nabla h(y)^\T y \leq -h(0) < 0$.
Since all feasible points satisfy $h(x) \leq 0$, they must satisfy the
infinite disjunction
\begin{equation}
    \label{eq:inf_disj}
    \bigvee_{y \in Y} \frac{\nabla h(y)^\T}{\nabla h(y)^\T y} x \geq 1.
\end{equation}

The maximum principle~\cite{Balas1979} implies that with
\begin{equation}
    \label{eq:cut_coeff}
    \alpha_j = \max_{y \in Y} \frac{\partial_j h(y)}{\nabla h(y)^\T y},
\end{equation}
the cut $\sum_j \alpha_j x_j \geq 1$ is valid.
We remark that the maximum exists, since the concavity of $h$ implies that
for $y \in Y$, $h(e_j) \leq \partial_j h(y) - \nabla h(y)^\T y$.
This implies, together with $\nabla h(y)^\T y \leq -h(0) < 0$, that
$\tfrac{\partial_j h(y)}{\nabla h(y)^\T y} \leq 1 + \tfrac{h(e_j)}{\nabla
h(y)^\T y}$.
If $h(e_j) \geq 0$, then $\tfrac{\partial_j h(y)}{\nabla h(y)^\T y} \leq 1$.
Otherwise, $\tfrac{\partial_j h(y)}{\nabla h(y)^\T y} \leq 1 -
\tfrac{h(e_j)}{h(0)}$.

The application of monoidal strengthening~\cite[Theorem 3]{BalasJeroslow1980} to
a valid disjunction $\bigvee_i \alpha^i x \geq 1$ requires the existence of
bounds $\beta_i$ such that $\alpha^i x \geq \beta$ is valid for every feasible
point.
Let $\beta(y)$ be such a bound for~\eqref{eq:inf_disj}.
An example of $\beta(y)$ is
\[
    \beta(y) = \min_{x \in [0,u]} \frac{\nabla h(y)^\T x}{\nabla h(y)^\T y}.
\]

\begin{remark}
    \label{rmk:beta_bound}
    If $\beta(y) \geq 1$, then $\nabla h(y)^\T x / \nabla h(y)^\T y \geq 1$ is
    redundant and can be removed from~\eqref{eq:inf_disj}.
    Therefore, we can assume without loss of generality that $\beta(y) < 1$.
\end{remark}

The strengthening derives from the fact that a new disjunction can be obtained
from~\eqref{eq:inf_disj} and, with it, a new disjunctive cut.
The disjunction on the following Lemma is trivially satisfied, but provides the
basis for building non-trivial new disjunctions.
\begin{lemma}
    \label{lemma:strength}
    Every $x \geq 0$ that satisfies~\eqref{eq:inf_disj}, also satisfies
    \begin{equation}
        \label{eq:pre_strengthening}
        \bigvee_{y \in Y}
        \frac{\nabla h(y)^\T x}{\nabla h(y)^\T y} + z(y) (1 - \beta(y)) \geq 1,
    \end{equation}
    where $z \colon Y \to \mathbb{Z}$ is such that $z \equiv 0$ or there is a
    $y_0 \in Y$ for which $z(y_0) > 0$.
\end{lemma}
\begin{proof}
    If $z \equiv 0$, then~\eqref{eq:pre_strengthening} reduces
    to~\eqref{eq:inf_disj}.

    Otherwise, let $y_0 \in Y$ such that $z(y_0) > 0$, that is, $z(y_0) \geq 1$.
    By \Cref{rmk:beta_bound}, for every $y \in Y$, it holds $1 - \beta(y) > 0$,
    and so
    \[
        z(y_0) (1 - \beta(y_0)) \geq 1 - \beta(y_0).
    \]
    Therefore, $\beta(y_0) \geq 1 - z(y_0) (1 - \beta(y_0))$.
    Since every $x \geq 0$ satisfying~\eqref{eq:inf_disj} satisfies
    $\frac{\nabla h(y_0)^\T x}{\nabla h(y_0)^\T y_0} \geq \beta(y_0)$, we
    conclude that $\frac{\nabla h(y_0)^\T x}{\nabla h(y_0)^\T y_0} + z(y_0) (1 -
    \beta(y_0)) \geq 1$ holds.
\end{proof}

\begin{remark}
    \label{rmk:unbounded_beta}
    Even if some disjunctive terms have no lower bound, that is, $\beta(y) =
    -\infty$ for $y \in Y' \subseteq Y$, \Cref{lemma:strength} still holds
    if, additionally, $z(y) = 0$ for all $y \in Y'$.
    This means that we are not using that disjunction for the strengthening.
    In particular, if for some variable $x_j$, $\alpha_j$ is defined by some
    $y \in Y'$, then this cut coefficient cannot be improved.
\end{remark}

Assume now that $x_k \in \mathbb{Z}$ for every $k \in K \subseteq \{1, \ldots,
n\}$.
To construct a new disjunction, we need to find a set of functions $M$ such that
for any choice of $m^k \in M$ and any feasible assignment of $x_k$, $\sum_{k \in
K} x_k m^k(y)$ satisfies the conditions of \Cref{lemma:strength}, that is, it
must be in
\[
    Z = \{ z \colon Y \to \mathbb{Z} : z \equiv 0 \vee \exists y \in Y, z(y) > 0 \}.
\]
Once such a family of functions has been identified, the cut
$\sum_j \gamma_j x_j \geq 1$ with $\gamma_j = \alpha_j$ if $j \notin K$,
and
\begin{equation}
    \label{eq:strong_cut_coeff}
    \gamma_k = \inf_{m \in M} \max_{y \in Y}
    \frac{\partial_k h(y)}{\nabla h(y)^\T y} + m(y) (1 - \beta(y))
    \quad \text{for } k \in K,
\end{equation}
is valid and at least as strong as~\eqref{eq:cut_coeff}.
Any $M \subseteq Z$ such that $(M,+)$ is a monoid, that is, $0 \in M$ and $M$ is
closed under addition can be used in~\eqref{eq:strong_cut_coeff}.

\begin{remark}
    \label{rmk:relation_bj80}
    This is exactly what is happening in~\cite[Theorem 3]{BalasJeroslow1980}.
    Indeed, in the finite case, that is, when $Y$ is finite, Balas and Jeroslow
    considered $M = \{ m \in \mathbb{Z}^Y : \sum_{y \in Y} m_y \geq 0\}$.
    Clearly, $(M,+)$ is a monoid and $M \subseteq Z$.
    Therefore, \Cref{lemma:strength} implies that
    $\bigvee_{y \in Y} \alpha^y x + \sum_k m^k_y x_k (1 - \beta_y) \geq 1$
    is valid for any choice of $m^k \in M$, which in turn
    implies~\cite[Theorem 3]{BalasJeroslow1980}.\\
    For an application that uses a different monoid see~\cite{BalasQualizza2012}.
\end{remark}

The question that remains is how to choose $M$.
For example, the monoid $M = \{ m \colon Y \to \mathbb{Z} : m \text{ has finite
support and } \sum_{y \in Y} m(y) \geq 0 \}$ is an obvious candidate for $M$.
However, the problem is how to optimize over such an $M$, see~\eqref{eq:strong_cut_coeff}.

We circumvent this problem by considering only one integer variable at a time.
Fix $k \in K$.
In this setting we can use $Z$ as $M$, which is \emph{not} a monoid.
Indeed, if $z \in Z$, then $x_k z \in Z$ for any $x_k \in \mathbb{Z}_+$.
The advantage of using $Z$ is that the solution of~\eqref{eq:strong_cut_coeff}
is easy to characterize.

With $M = Z$, the cut coefficients~\eqref{eq:strong_cut_coeff} of all variables
are the same as~\eqref{eq:cut_coeff} except for $x_k$.
The cut coefficient of $x_k$ is given by
\[
    \inf_{z \in Z} \max_{y \in Y} \frac{\partial_k h(y)}{\nabla h(y)^\T y} +
    z(y) (1 - \beta(y)).
\]

To compute this coefficient, observe that one would like to have $z(y) < 0$ for
points $y$ such that the objective function of~\eqref{eq:cut_coeff} is large.
However, $z$ must be positive for at least one point.
Therefore,
\begin{equation*}
    \min_{y \in Y} \frac{\partial_k h(y)}{\nabla h(y)^\T y} + (1 - \beta(y))
\end{equation*}
is the best coefficient we can hope for if $z \not\equiv 0$.
This coefficient can be achieved by
\begin{equation}
    \label{eq:optimal_z}
    z(y) = \begin{cases}
        1, &\text{ if } y \in \argmin_{y \in Y} \frac{\partial_k h(y)}{\nabla
        h(y)^\T y} + (1 - \beta(y)), \\
        -L, &\text{ otherwise }
    \end{cases}
\end{equation}
where $L > 0$ is sufficiently large.

Summarizing, we can obtain the following cut:
\begin{equation}
    \label{eq:final_coef}
    \alpha_j =
    \begin{cases}
        \max_{y \in Y} \frac{\partial_j h(y)}{\nabla h(y)^\T y}
        &\text{ if } j \neq k \\
        \min\{
            \max_{y \in Y} \frac{\partial_j h(y)}{\nabla h(y)^\T y},
            \min_{y \in Y} \frac{\partial_j h(y)}{\nabla h(y)^\T y} + (1 - \beta(y))
        \}
        &\text{ if } j = k
    \end{cases}
\end{equation}

\begin{remark}
    \label{rmk:more_integers}
    Let $z^k \in Z$ be given by~\eqref{eq:optimal_z} for each $k \in K$.
    Assume there is a subset $K_0 \subseteq K$ and a monoid $M \subseteq Z$
    such that $z^k \in M$ for every $k \in K_0$.
    Then, the strengthening can be applied to all $x_k$ for $k \in K_0$.

    Alternatively, if there is a constraint enforcing that at most one of the
    $x_k$ can be non-zero for $k \in K_0$, e.g., $\sum_{k \in K} x_k \leq 1$,
    then the strengthening can be applied to all $x_k$ for $k \in K_0$.
\end{remark}

\begin{example}
    \label{ex:monoidal}
    Consider the constraint $\{ x \in \{0,1,2\} \times [0,5] : h(x) \leq 0\}$,
    where $h(x_1, x_2) = -10x_1^2 - 1/2 x_2^2 + 2x_1x_2 + 4$, see \Cref{fig:ex3}.
    The IC is given by $\sqrt{5/2} x_1 + 1/(2 \sqrt{2}) x_2 \geq 1$.
    Note that $(1/\sqrt{10}, \sqrt{10}) \in Y$ and yields the term $1/\sqrt{10}
    x_2 \geq 1$ in~\eqref{eq:inf_disj}.
    Since $x_2 \geq 0$, $\beta (1/\sqrt{10}, \sqrt{10}) = 0$.
    Hence,~\eqref{eq:final_coef} yields $\alpha_1 \leq \min\{\sqrt{5/2}, 1\} =
    1$ and the strengthened inequality is $x_1 + 1/(2 \sqrt{2}) x_2 \geq 1$.
    \begin{figure}
        \includegraphics[width=\textwidth]{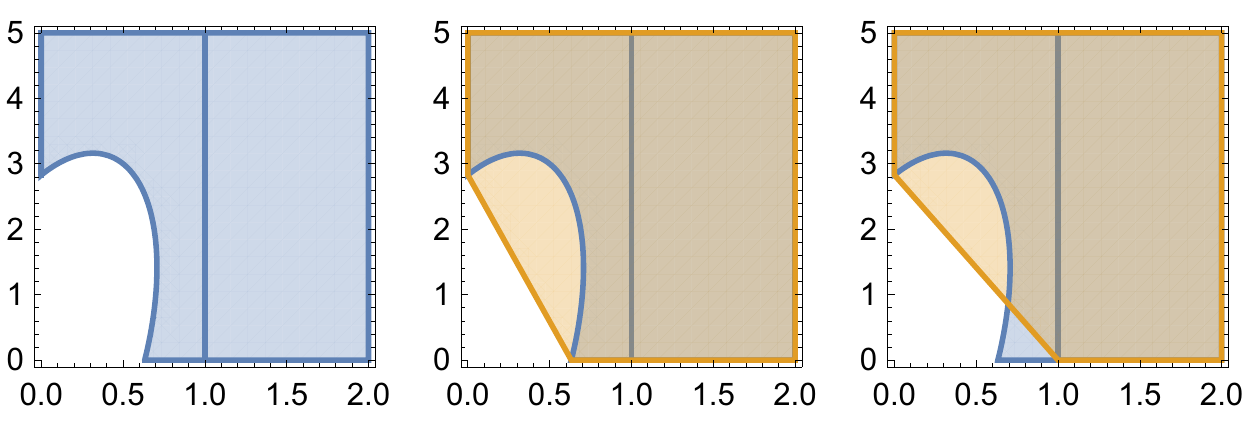}
        \caption{The feasible region $\{ x \in \{0,1,2\} \times [0,5] : h(x)
            \leq 0\}$ from \Cref{ex:monoidal} (left), the IC (middle), and the
            strengthened cut (right).
        } \label{fig:ex3}
    \end{figure}
\end{example}


\section{Conclusions}
\label{sec:conclusions}
We have introduced a procedure to generate concave underestimators of factorable
functions, which can be used to generate intersection cuts, together with
two strengthening procedures.

It remains to be seen the practical performance of these intersection cuts.
We expect that its generation is cheaper than the generation of disjunctive
cuts, given that there is no need to solve an LP.
As for the strengthening procedures, they might be too expensive to be of
practical use.
An alternative is to construct a polyhedral inner approximation of the $S$-free
set and use monoidal strengthening in the finite setting.
However, in this case, the strengthening proposed in \Cref{sec:strength}
has no effect.
Nonetheless, as far as the author knows, this has been the first application of
monoidal strengthening that is able to exploit further problem structure such as
demonstrated in \Cref{rmk:more_integers} and it might be interesting to
investigate further.

\paragraph{Acknowledgments} 
The author is indebted to Stefan Vigerske and Franziska Schl\"osser for many
helpful discussions and comments that improved the manuscript.
He would also like to thank Sven Wiese, Ambros Gleixner, Dan Steffy and Juan
Pablo Vielma for helpful discussions, and Leon Eifler, Daniel Rehfeldt for
comments that improved the manuscript.

\bibliographystyle{abbrv}
\bibliography{intercuts}

\begin{thebibliography}{10}

\bibitem{AchterbergWunderling2013}
T.~Achterberg and R.~Wunderling.
\newblock Mixed integer programming: Analyzing 12 years of progress.
\newblock In {\em Facets of Combinatorial Optimization}, pages 449--481.
  Springer Berlin Heidelberg, 2013.

\bibitem{Balas1971}
E.~Balas.
\newblock Intersection cuts{\textemdash}a new type of cutting planes for
  integer programming.
\newblock {\em Operations Research}, 19(1):19--39, feb 1971.

\bibitem{Balas1979}
E.~Balas.
\newblock Disjunctive programming.
\newblock In {\em Discrete Optimization {II}, Proceedings of the Advanced
  Research Institute on Discrete Optimization and Systems Applications of the
  Systems Science Panel of {NATO} and of the Discrete Optimization Symposium
  co-sponsored by {IBM} Canada and {SIAM} Banff, Aha. and Vancouver}, pages
  3--51. Elsevier {BV}, 1979.

\bibitem{BalasCeriaCornuejols1993}
E.~Balas, S.~Ceria, and G.~Cornu{\'{e}}jols.
\newblock A lift-and-project cutting plane algorithm for mixed 0{\textendash}1
  programs.
\newblock {\em Mathematical Programming}, 58(1-3):295--324, jan 1993.

\bibitem{BalasJeroslow1980}
E.~Balas and R.~G. Jeroslow.
\newblock Strengthening cuts for mixed integer programs.
\newblock {\em European Journal of Operational Research}, 4(4):224--234, apr
  1980.

\bibitem{BalasMargot2011}
E.~Balas and F.~Margot.
\newblock Generalized intersection cuts and a new cut generating paradigm.
\newblock {\em Mathematical Programming}, 137(1-2):19--35, aug 2011.

\bibitem{BalasQualizza2012}
E.~Balas and A.~Qualizza.
\newblock Monoidal cut strengthening revisited.
\newblock {\em Discrete Optimization}, 9(1):40--49, feb 2012.

\bibitem{BasuCornuejolsZambelli2011}
A.~Basu, G.~Cornu{\'e}jols, and G.~Zambelli.
\newblock Convex sets and minimal sublinear functions.
\newblock {\em Journal of Convex Analysis}, 18(2):427--432, 2011.

\bibitem{Belotti2011}
P.~Belotti.
\newblock Disjunctive cuts for nonconvex {MINLP}.
\newblock In {\em Mixed Integer Nonlinear Programming}, pages 117--144.
  Springer New York, nov 2011.

\bibitem{BelottiLeeLibertiMargotWaechter2009}
P.~Belotti, J.~Lee, L.~Liberti, F.~Margot, and A.~Wächter.
\newblock Branching and bounds tightening techniques for non-convex {MINLP}.
\newblock {\em Optimization Methods and Software}, 24(4-5):597--634, oct 2009.

\bibitem{BienstockChenMunoz2016}
D.~Bienstock, C.~Chen, and G.~Mu{\~{n}}oz.
\newblock Outer-product-free sets for polynomial optimization and oracle-based
  cuts.

\bibitem{BonamiLinderothLodi2011}
P.~Bonami, J.~Linderoth, and A.~Lodi.
\newblock Disjunctive cuts for mixed integer nonlinear programming problems.
\newblock {\em Progress in Combinatorial Optimization}, pages 521--541, 2011.

\bibitem{BoukouvalaMisenerFloudas2016}
F.~Boukouvala, R.~Misener, and C.~A. Floudas.
\newblock Global optimization advances in mixed-integer nonlinear programming,
  {MINLP}, and constrained derivative-free optimization, {CDFO}.
\newblock {\em European Journal of Operational Research}, 252(3):701--727, aug
  2016.

\bibitem{BrondstedRockafellar1965}
A.~Brondsted and R.~T. Rockafellar.
\newblock On the subdifferentiability of convex functions.
\newblock {\em Proceedings of the American Mathematical Society}, 16(4):605,
  aug 1965.

\bibitem{BuchheimDAmbrosio2016}
C.~Buchheim and C.~D'Ambrosio.
\newblock Monomial-wise optimal separable underestimators for mixed-integer
  polynomial optimization.
\newblock {\em Journal of Global Optimization}, 67(4):759--786, may 2016.

\bibitem{BuchheimTraversi2013}
C.~Buchheim and E.~Traversi.
\newblock Separable non-convex underestimators for binary quadratic
  programming.
\newblock In {\em Experimental Algorithms}, pages 236--247. Springer Berlin
  Heidelberg, 2013.

\bibitem{ConfortiCornuejolsDaniilidisLemarechalMalick2015}
M.~Conforti, G.~Cornu{\'{e}}jols, A.~Daniilidis, C.~Lemar{\'{e}}chal, and
  J.~Malick.
\newblock Cut-generating functions and {S}-free sets.
\newblock {\em Mathematics of Operations Research}, 40(2):276--391, may 2015.

\bibitem{ConfortiCornuejolsZambelli2011a}
M.~Conforti, G.~Cornu{\'{e}}jols, and G.~Zambelli.
\newblock Corner polyhedron and intersection cuts.
\newblock {\em Surveys in Operations Research and Management Science},
  16(2):105--120, jul 2011.

\bibitem{ConfortiCornuejolsZambelli2011}
M.~Conforti, G.~Cornu{\'{e}}jols, and G.~Zambelli.
\newblock A geometric perspective on lifting.
\newblock {\em Operations Research}, 59(3):569--577, jun 2011.

\bibitem{DeyWolsey2010}
S.~S. Dey and L.~A. Wolsey.
\newblock Two row mixed-integer cuts via lifting.
\newblock {\em Mathematical Programming}, 124(1-2):143--174, may 2010.

\bibitem{Glover1973}
F.~Glover.
\newblock Convexity cuts and cut search.
\newblock {\em Operations Research}, 21(1):123--134, feb 1973.

\bibitem{Glover1974}
F.~Glover.
\newblock Polyhedral convexity cuts and negative edge extensions.
\newblock {\em Zeitschrift für Operations Research}, 18(5):181--186, oct 1974.

\bibitem{Hasan2018}
M.~M.~F. Hasan.
\newblock An edge-concave underestimator for the global optimization of
  twice-differentiable nonconvex problems.
\newblock {\em Journal of Global Optimization}, 71(4):735--752, mar 2018.

\bibitem{Khamisov1999}
O.~Khamisov.
\newblock On optimization properties of functions, with a concave minorant.
\newblock {\em Journal of Global Optimization}, 14(1):79--101, 1999.

\bibitem{KilincSahinidis2017}
M.~R. K{\i}l{\i}n{\c{c}} and N.~V. Sahinidis.
\newblock Exploiting integrality in the global optimization of mixed-integer
  nonlinear programming problems with {BARON}.
\newblock {\em Optimization Methods and Software}, 33(3):540--562, jul 2017.

\bibitem{LinSchrage2009}
Y.~Lin and L.~Schrage.
\newblock The global solver in the {LINDO} {API}.
\newblock {\em Optimization Methods and Software}, 24(4-5):657--668, oct 2009.

\bibitem{McCormick1976}
G.~P. McCormick.
\newblock Computability of global solutions to factorable nonconvex programs:
  Part i {\textemdash} convex underestimating problems.
\newblock {\em Mathematical Programming}, 10(1):147--175, dec 1976.

\bibitem{MisenerFloudas2014}
R.~Misener and C.~A. Floudas.
\newblock {ANTIGONE}: Algorithms for {coNTinuous} / {Integer Global
  Optimization of Nonlinear Equations}.
\newblock {\em Journal of Global Optimization}, 59(2-3):503--526, mar 2014.

\bibitem{ModaresiKilincVielma2015}
S.~Modaresi, M.~R. K{\i}l{\i}n{\c{c}}, and J.~P. Vielma.
\newblock Intersection cuts for nonlinear integer programming: convexification
  techniques for structured sets.
\newblock {\em Mathematical Programming}, 155(1-2):575--611, feb 2015.

\bibitem{Porembski1999}
M.~Porembski.
\newblock How to extend the concept of convexity cuts to derive deeper cutting
  planes.
\newblock {\em Journal of Global Optimization}, 15(4):371--404, 1999.

\bibitem{Porembski2001}
M.~Porembski.
\newblock Finitely convergent cutting planes for concave minimization.
\newblock {\em Journal of Global Optimization}, 20(2):109--132, 2001.

\bibitem{SaxenaBonamiLee2010}
A.~Saxena, P.~Bonami, and J.~Lee.
\newblock Convex relaxations of non-convex mixed integer quadratically
  constrained programs: extended formulations.
\newblock {\em Mathematical Programming}, 124(1-2):383--411, may 2010.

\bibitem{SaxenaBonamiLee2010a}
A.~Saxena, P.~Bonami, and J.~Lee.
\newblock Convex relaxations of non-convex mixed integer quadratically
  constrained programs: projected formulations.
\newblock {\em Mathematical Programming}, 130(2):359--413, mar 2010.

\bibitem{SenSherali1986}
S.~Sen and H.~D. Sherali.
\newblock Facet inequalities from simple disjunctions in cutting plane theory.
\newblock {\em Mathematical Programming}, 34(1):72--83, jan 1986.

\bibitem{SenSherali1987}
S.~Sen and H.~D. Sherali.
\newblock Nondifferentiable reverse convex programs and facetial convexity cuts
  via a disjunctive characterization.
\newblock {\em Mathematical Programming}, 37(2):169--183, jun 1987.

\bibitem{TawarmalaniSahinidis2002}
M.~Tawarmalani and N.~V. Sahinidis.
\newblock Convex extensions and envelopes of lower semi-continuous functions.
\newblock {\em Mathematical Programming}, 93(2):247--263, dec 2002.

\bibitem{TawarmalaniSahinidis2005}
M.~Tawarmalani and N.~V. Sahinidis.
\newblock A polyhedral branch-and-cut approach to global optimization.
\newblock {\em Mathematical Programming}, 103(2):225--249, may 2005.

\bibitem{TsoukalasMitsos2014}
A.~Tsoukalas and A.~Mitsos.
\newblock Multivariate {McCormick} relaxations.
\newblock {\em Journal of Global Optimization}, 59(2-3):633--662, apr 2014.

\bibitem{Tuy1964}
H.~Tuy.
\newblock Concave programming with linear constraints.
\newblock In {\em Doklady Akademii Nauk}, volume 159, pages 32--35. Russian
  Academy of Sciences, 1964.

\bibitem{VigerskeGleixner2017}
S.~Vigerske and A.~Gleixner.
\newblock {SCIP}: global optimization of mixed-integer nonlinear programs in a
  branch-and-cut framework.
\newblock {\em Optimization Methods and Software}, 33(3):563--593, jun 2017.

\bibitem{Wiese2016}
S.~Wiese.
\newblock {\em On the interplay of Mixed Integer Linear, Mixed Integer
  Nonlinear and Constraint Programming}.
\newblock PhD thesis, 2016.

\end{thebibliography}

\end{document}